%% file: CensusCayley.tex
\documentclass{amsart}
\usepackage{amssymb}
\usepackage[utf8]{inputenc}
\usepackage[T1]{fontenc}
\usepackage[english]{babel}
\usepackage{csquotes}
\usepackage{comment}
\usepackage{xargs}
\usepackage{graphicx, color}
\setkeys{Gin}{draft=false}
\usepackage[figurewithin=none]{caption,subcaption} 
\graphicspath{{./}}
\usepackage{enumitem}
\setlist[enumerate,1]{label={(\roman*)}}

\usepackage[bb=dsserif]{mathalpha}

\usepackage[noEnd=true,beginComment=//~,beginLComment=/*~, endLComment=~*/]{algpseudocodex}
\usepackage[section]{algorithm}
\algnewcommand{\Break}{\textbf{break}}
\algnewcommand{\Record}{\textbf{record}}
\algnewcommand{\AND}{\textbf{and} }

\usepackage[pagewise,modulo,mathlines]{lineno}

 \usepackage{hyperref}
 \hypersetup{
 	hidelinks,
   colorlinks  = true,    
   urlcolor    = blue,    
   linkcolor   = blue,    
   citecolor   = blue,      
   pdfauthor   = { },%
   pdfsubject  = {Census of Cayley graphs},%
   pdftitle    = {},  %
   breaklinks=true,
}

\usepackage[capitalise,noabbrev, nameinlink,]{cleveref}
\usepackage{tikz-cd}
\tikzcdset{%
row sep=tiny,%
arrows={shorten=-4pt}
}

\usepackage[loadshadowlibrary,shadow, colorinlistoftodos,textsize=tiny,obeyFinal]{todonotes}

\setuptodonotes{fancyline, backgroundcolor=gray!10,bordercolor=gray}

\newcommandx{\inline}[2][1=]{\todo[inline, #1]{#2}}
\newcounter{hw}
\newcommandx{\homework}[2][1=]{\todo[inline, caption={Homework \thehw} #1]{\stepcounter{hw} #2}}

\newcommandx{\rhysLong}[2][1=Long todo, usedefault]{\todo[inline, color=green!25, caption={\textbf{Rhys:} #1}]{\textbf{Rhys: #1}. #2}}
\newcommandx{\primozLong}[2][1=Long todo, usedefault]{\todo[inline,color=yellow!25, caption={\textbf{Primož:} #1}]{\textbf{Primož: #1}. #2} }
\newcommandx{\noteLong}[2][1=Long todo, usedefault]{\todo[inline,color=red!25, caption={\textbf{Note:} #1}]{\textbf{Note: #1}. #2} }

\makeatletter
    \providecommand\@dotsep{5}
\makeatother

\input{authorThm}

\input{authorCmd}

\usepackage[%
sorting=nyt, %
backend=biber, %
style=numeric,%
defernumbers=true,%
url=false,%
eprint=false,%
giveninits=true,
maxbibnames=9,%
isbn=false,%
sortcites = true,%
citestyle=numeric-comp,%
]{biblatex}
\begin{document}

\title{A census of Cayley graphs}

\author{Rhys J. Evans}
\address{\parbox{\linewidth}{Rhys J. Evans\\ School of Mathematics and Statistics, University of Sydney\\ Camperdown, Sydney, NSW 2006, Australia}}%

\email{rhysjevans00@gmail.com}

\author{Primo\v{z} Poto\v{c}nik}
\address{\parbox{\linewidth}{Primo\v{z} Poto\v{c}nik\\ 
Faculty of Mathematics and Physics, University of Ljubljana,\\
Jadranska ulica 21, 1000 Ljubljana, Slovenia;\\
Also affiliated with\\
Institute of Mathematics, Physics and Mechanics\\ Jadranska ulica 19, 1000 Ljubljana, Slovenia}}%

\email{primoz.potocnik@fmf.uni-lj.si}

\maketitle

\input{paper2}

\setcounter{biburlnumpenalty}{100}
\printbibliography
\end{document}

%% file: authorThm.tex
\theoremstyle{plain}
\newtheorem{thm}{Theorem}[section]

\newtheorem{lem}[thm]{Lemma}

\newtheorem*{conjecture*}{Conjecture}
\newtheorem{problem}{Problem}

\theoremstyle{definition}

\newtheorem{definition}[thm]{Definition}

\theoremstyle{remark}

\numberwithin{equation}{section}

\Crefname{thm}{Theorem}{Theorems}
\Crefname{theorem}{Theorem}{Theorems}
\Crefname{lem}{Lemma}{Lemmas}
\Crefname{lemma}{Lemma}{Lemmas}
\Crefname{prop}{Proposition}{Propositions}
\Crefname{proposition}{Proposition}{Propositions}
\Crefname{coro}{Corollary}{Corollaries}
\Crefname{corollary}{Corollary}{Corollaries}
\Crefname{defn}{Definition}{Definitions}
\Crefname{definition}{Definition}{Definitions}
\Crefname{exam}{Example}{Examples}
\Crefname{eg}{Example}{Examples}
\Crefname{example}{Example}{Examples}
\Crefname{note}{Note}{Notes}
\Crefname{rem}{Remark}{Remarks}
\Crefname{remark}{Remark}{Remarks}

%% file: authorCmd.tex
\usepackage{xargs}

\renewcommand{\leq}{\leqslant} %
\renewcommand{\geq}{\geqslant}
\renewcommand{\epsilon}{\varepsilon} %
\renewcommand{\subset}{\subseteq} %

\renewcommand{\{}{\lbrace}
\renewcommand{\}}{\rbrace}
 
\newcommand{\cA}{\mathcal{A}}

\newcommand{\cG}{\mathcal{G}}

\newcommand{\cN}{\mathcal{N}}

\newcommand{\cQ}{\mathcal{Q}}
\newcommand{\cS}{\mathcal{S}}

\DeclareMathOperator{\aut}{Aut} 
\DeclareMathOperator{\cay}{Cay}

\newcommand{\ZZ}{\mathbb{Z}}

\newcommand{\tgroup}{\mathbb{1}}

\newcommand{\gen}[1]{\langle #1 \rangle}

\newcommand{\Orbits}[2]{\mathrm{Orbits}_{#1}({#2})}

\newcommand{\GL}[2]{\mathrm{GL}_{#1}(#2)}

%% file: paper2.tex
\begin{abstract}
Given positive integers $k$ and $n$, we present methods to construct all groups of order at most $n$ that contain a Cayley set of size $k$, and to enumerate the Cayley sets of order $k$ in a given group, up to the action of the automorphism group. We use these methods to generate complete lists of pairwise nonisomorphic 3-valent Cayley graphs with at most 5000 vertices and 4-valent Cayley graphs with at most 1025 vertices. 
\end{abstract}

\section{Introduction}\label{s:intro}

In this article we derive exhaustive generation methods used for the enumeration of Cayley graphs of small order and valency. Moreover, this article serves to describe and document recently obtained datasets of {\em all 3-valent Cayley graphs of order up to $5{,}000$} and of {\em all 4-valent Cayley graphs of order up to 1025}.

Compiling datasets (or `censuses' as they are often called) of symmetric graphs (as well other mathematical objects) has a long history, going back at least to 1930's when Foster \cite{Fos32} started collecting examples of highly symmetrical $3$-regular graphs (or `geometric circuits of electrical networks', as he called them). His work evolved over the years and culminated in a book, published in 1988, entitled {\em The Foster Census} \cite{FosterCensus}, which contained all but a few (connected, simple) $3$-valent graphs of order up to $512$ whose automorphism groups act transitively on the corresponding sets of {\em arcs} (that is, ordered pairs of adjacent vertices). This work was superseded in 2002 by Conder and Dobcs\'{a}nyi~\cite{ConderFosterCensus}, who used  results of Tutte's seminal paper~\cite{T_1947} and their newly developed algorithm for finding smooth small quotients of finitely generated groups to produce a first complete computer generated list of all $3$-valent arc-transitive graphs of order up to $768$. This list was later extended several times and now goes up to order $10{,}000$ \cite{ConPot25}.

At the same time as Foster worked on compiling his census of cubic arc-transitive graphs, Coxeter, Frucht and Powers \cite{zero} considered the class of $3$-valent graphs whose automorphism groups act regularly (transitively and with trivial stabiliser) on the corresponding vertex-sets. In modern language, they were considering $3$-valent {\em graphical regular representations}, a subclass of $3$-valent Cayley graphs. In the 1980's and 1990's, different datasets of small Cayley graphs of small valency (most notably valency $3$) were compiled by Royle and McKay, using a mixture of theoretical analysis and exhaustive computer enumeration. Most of these datasets, which were published online on Royle's university webpage, seem to be unavailable due to the change of the university policy \cite{RoylePersonal}. This is on one hand very unfortunate (though mainly for historical reasons, as most of these datasets are now included in more comprehensive censuses) and on the other hand emphasises the need for a more systematic publication and custodianship of mathematical data. 
By documenting our work in a form of an article, together with archiving the raw data on trusted repositories \cite{zenodo_2024}, we tried to make a step in this direction.

This paper is organised as follows. In Section \ref{s:prelims} we present preliminary results and new definitions. In particular, we define an \emph{$(a,b)$-valent Cayley set} of a group $G$ as an inverse-closed generating set $S$ of $G$ that does not contain the identity, and such that $S$ contains $a$ involutions and $b$ noninvolutions. This is followed by an outline of our enumerative approaches in Subsection \ref{ss:outline}, where we describe how to use the generation of $(a,b)$-valent Cayley sets for a finite set of pairs $(a,b)$ to enumerate the $k$-valent Cayley graphs of small order. In Section \ref{s:subvalent_groups} we present methods for the generation of all groups up to a given order which have $(a,b)$-valent Cayley sets. In Section \ref{s:cayley_sets} we give methods for the generation of all $(a,b)$-valent Cayley sets of a given group $G$, up to the action of $\aut(G)$. In Section \ref{s:extract_cayley} we describe the final step of the enumeration of small $k$-valent Cayley graphs, where we construct the Cayley graph defined by each Cayley set we have generated from the methods of Section \ref{s:cayley_sets} and reduce the resulting list up to graph isomorphism. 

In Section \ref{s:stats}, we present numerical data found in our enumerations of 3-valent and 4-valent Cayley graphs. Observing the trends in our data, we give several interesting open problems related to the automorphism group (Subsection \ref{ss:aut_data}), girth (Subsection \ref{ss:girth_data}) and bipartiteness (Subsection \ref{ss:bipartite_data}) of Cayley graphs of a given fixed valency.

\section{Preliminaries}\label{s:prelims}

In this paper, all graphs will be simple (i.e, undirected with no loops or multiple edges). The automorphism group of a graph $\Gamma$ is denoted by $\aut(\Gamma)$. Similarly, the automorphism group of a group $G$ is denoted by $\aut(G)$. 

Let $G$ be a group. A set $S\subset G$ such that $1\notin S$ is a \emph{Cayley set} of $G$ if $S$ is inverse-closed and generates $G$. Given a Cayley set $S$ of $G$, the \emph{Cayley graph} $\cay(G,S)$ is the graph with vertex-set $G$, such that distinct vertices $x,y$ are adjacent if and only if $xy^{-1}\in S$. In this case, $S$ is also called the \emph{Cayley set} of $\cay(G,S)$. It is important to be aware that in our definition, all Cayley graphs are connected. For our purposes, we introduce new notation relating Cayley graphs with their defining groups and Cayley sets.

For pairs of non-negative integers $(a,b),(a',b')$, we say that $(a',b')\preccurlyeq(a,b)$ if and only if $b'\leq b$ and $a'\leq a+(b-b')/2$. This gives a partial order on the set of pairs of non-negative integers. We will be interested in the induced ordering of $\preccurlyeq$ on the set of Cayley sets of a group, where a Cayley set $S$ is associated to the pair $(a,b)$ where $S$ contains $a$ involutions and $b$ noninvolutions.

\begin{definition}[for Cayley sets]\label{def:subvalent_sets}
    Let $G$ be a group with Cayley set $S$, and $a,b,k$ be nonnegative integers. Then;
\begin{enumerate}
    \item $S$ is \emph{sub-$k$-valent} if $|S|\leq k$. $S$ is \emph{$k$-valent} (or has \emph{valency $k$}) if $|S|=k$.
    \item $S$ is  \emph{sub-$(a,b)$-valent} if it consists of $a'$ involutions and $b'$ noninvolutions, such that $(a',b')\preccurlyeq(a,b)$. $S$ is \emph{$(a,b)$-valent} (or has \emph{valency $(a,b)$}) if it consists of $a$ involutions and $b$ noninvolutions.
    \end{enumerate}  
\end{definition}

Note that the graph $\cay(G,S)$ is regular with valency $k$ if and only if $S$ is a $k$-valent Cayley set. Furthermore, if $S$ is an $(a,b)$-valent Cayley set, then $b$ is even and $S$ is $(a+b)$-valent. If $S$ is sub-$(a,b)$-valent, it must be $(a',b')$-valent for some $a',b'$ such that $a'+b'\leq a+b$, so $S$ is sub-$(a+b)$-valent. We will denote the set of pairs $(a,b)$ such that $a,b$ are nonnegative integers, $b$ is even and $a+b=k$ as $\Sigma_k$.

\begin{definition}[for groups]
A group $G$ is \emph{sub-$k$-valent} if there exists a $k$-valent Cayley set of $G$. We define $k$-valent, $(a,b)$-valent and sub-$(a,b)$-valent groups analogously to the definitions for Cayley sets.    
\end{definition}

Therefore, there is a (connected) $k$-valent Cayley graph defined on $G$ if and only if $G$ is $k$-valent. Throughout the remainder of the paper, we will often replace `3-valent' and `4-valent' with `cubic' and `quartic', respectively. For example, the terms sub-3-valent and subcubic will be synonymous.

It is important to note that the class of sub-$(a,b)$-valent groups is closed under taking quotients. This property was the motivation behind the definitions of $\preccurlyeq$ and sub-$(a,b)$-valent Cayley sets.

\begin{lem}\label{lem:quotient_closed}
    Let $G$ be a group and $N\unlhd G$. Then;
    \begin{enumerate}
        \item If $S$ is a sub-$(a,b)$-valent Cayley set of $G$, $S$ also defines a sub-$(a,b)$-valent Cayley set of $G/N$.
        \item If $G$ is a sub-$(a,b)$-valent group, $G/N$ is also a sub-$(a,b)$-valent group.
    \end{enumerate}
\end{lem}
\begin{proof}
    (i) Let $\rho:G\to G/N$ be the natural homomorphism and $S_a,S_b$ be the set of involutions and noninvolutions in $S$, respectively. First note that $\rho(S)$ generates $G/N$ as $S$ generates $G$, and $\rho(S)$ is inverse-closed. 
    
    For all $s\in S$, the order of $\rho(s)$ divides the order of $s$. Therefore, $T_a:=\rho(S_a)\setminus 1$ is a set of involutions. Let $R_a,R_b$ be the set of involutions and noninvolutions of $\rho(S_b)\setminus 1$, respectively. For any element $s\in S_b$ such that $\rho(s)\in R_a$, $\rho(s^{-1})=\rho(s)$  and $s^{-1}\in S\setminus\{s\}$. Therefore, $b=|S_b|\geq 2|R_a|+|R_b|$. 

    Let $a':=|T_a\cup R_a|$ and $b':=|R_b|$. It follows that $\rho(S)\setminus 1$ is the disjoint union of the set of involutions $T_a\cup R_a$ and set of noninvolutions $R_b$, and defines an $(a',b')$-valent Cayley set of $G/N$.  By the above inequality and $|T_a|\leq |\rho(S_a)|\leq |S_a|=a$, we see that $$a'=|T_a\cup R_a|\leq |T_a|+|R_a|\leq a+(b-b')/2.$$ Therefore, $(a',b')\preccurlyeq (a,b)$.

    (ii) This follows from the definitions and part (i).
\end{proof}

Lemma \ref{lem:quotient_closed} also proves that sub-$k$-valent groups are closed under taking quotients. One computationally difficult subclass is the sub-$(k,0)$-valent groups. On the other hand, the subclass of groups of odd order are not $(a,b)$-valent for any nonzero value of $a$, resulting in faster computations relative to their order. 

The methods found in Section \ref{s:subvalent_groups} take advantage of the quotient-closed property, and can often be applied to quotient-closed subclasses of sub-$k$-valent groups. When a subclass is known to pose a particular problem (e.g., $2$-groups), we use additional theoretical results to improve the implementation of our methods (Subsections \ref{ss:min_normal} and \ref{ss:extension_methods}).

The right regular action of $G$ induces a regular group of automorphisms of $\cay(G,S)$, so $\cay(G,S)$ is a vertex-transitive graph. Conversely, by a theorem attributed to Sabidussi, any graph with a regular group of automorphisms is a Cayley graph of some group and Cayley set. The action of $\aut(G)$ on the elements of $G$ induces an action on the Cayley sets of $G$, which can be further restricted to the Cayley sets of a given valency $k$ or $(a,b)$. Given an automorphism $\phi\in \aut(G)$, $\phi$ naturally induces a graph isomorphism from $\cay(G,S)$ to $\cay(G,\phi(S))$. Therefore, each orbit of the action of $\aut(G)$ on the Cayley sets of $G$ defines a set of pairwise isomorphic graphs. Moreover, if $\phi \in \aut(G)$ fixes a Cayley set $S$ setwise, then $\phi$ is an automorphism of the Cayley graph $\cay(G,S)$. We define $\aut(G,S)$ as the set of all such $\phi$, and note that
$\aut(G,S) = \aut(\cay(G,S)) \cap \aut(G)$.

\subsection{Outline of the computational approach}\label{ss:outline}

Let $k,n$ be integers such that $k<n$, and $(a,b)\in \Sigma_k$. We present a series of methods used to attain all sub-$(a,b)$-valent groups of order at most $n$, together with $(a,b)$-valent Cayley sets. The relaxation to sub-$(a,b)$-valent groups is necessary when enumerating the $(a,b)$-valent groups of order $n$, as our methods require the class of groups to be closed under taking quotients.

As previously noted, all Cayley graphs of valency at most $k$ can be constructed from a sub-$(a,b)$-valent Cayley set $S$ of a sub-$(a,b)$-valent group $G$, where $(a,b)\in \Sigma_k$. Furthermore, any two representatives of an $\aut(G)$-orbit of Cayley sets of $G$ define isomorphic Cayley graphs. Therefore, the enumeration of Cayley graphs of valency at most $k$ and order at most $n$ (up to graph isomorphism) is divided into the following steps:

\begin{enumerate}[label=(\Roman*),itemsep=.3em,topsep=.7em]\label{ls:general_approach}
	\item\label{ls:general_approach_i} For each $(a,b)\in \Sigma_k$, construct a list $\cG(n,a,b)$ of pairwise nonisomorphic groups which contains all sub-$(a,b)$-valent groups of order at most $n$, but possibly also some groups that are not sub-$(a,b)$-valent (Section \ref{s:subvalent_groups}).
	\item\label{ls:general_approach_ii} For each $(a,b)\in \Sigma_k$ and group $G\in \cG(n,a,b)$, construct a list $\cS(G,a,b)$ of $(a,b)$-valent Cayley sets of $G$ that contains at least one representative of each $\aut(G)$-orbit (Section \ref{s:cayley_sets}).
	\item\label{ls:general_approach_iii} For each $(a,b)\in \Sigma_k$,  $G\in\cG(n,a,b)$ and $S\in\cS(G,a,b)$, construct $\cay(G,S)$ and reduce the resulting list of graphs up to isomorphism (Section \ref{s:extract_cayley}).
\end{enumerate}

Note that we allow  $\cG(n,a,b)$ to contain groups that are not $k$-valent. This is because our methods initially construct a list of extensions of sub-$(a,b)$-valent group, which may not be $(a,b)$-valent. At this point, the most efficient method we are aware of for recognising when a group $G$ is not $(a,b)$-valent are those applied in step \ref{ls:general_approach_ii}. The algorithms will be applicable to each group $G\in\cG(n,a,b)$, and return nothing if and only if $G$ is not $(a,b)$-valent (i.e, $\cS(G,a,b)=\emptyset$). Therefore, applying step \ref{ls:general_approach_ii} to each group in $\cG(n,a,b)$ will result in the identification of the $(a,b)$-valent groups of order at most $n$.  

Moreover, while we reduce the groups modulo isomorphism within each $(a,b)$-valent subclass for $(a,b)\in \Sigma_k$, we do not reduce them between different $(a,b)$-valent subclasses at this point. This is because for certain orders, such reductions are computationally difficult. However, we are guaranteed to construct at least one representative for every isomorphism class of sub-$k$-valent groups of order at most $n$. 

Similarly, we only require that $\cS(G,a,b)$ contains at least one representative of each $\aut(G)$-orbit on $(a,b)$-valent Cayley sets of $G$. By the above discussion, the set of Cayley graphs defined by members of $\cS(G,k)$ will cover all possible isomorphism classes of Cayley graphs of valency $(a,b)$ on $G$. Allowing multiple orbit representatives in $\cS(G,a,b)$ will only introduce isomorphic copies of a Cayley graph defined by Cayley sets in the same $\aut(G)$-orbit. Therefore, we can construct all Cayley graphs of valency at most $k$ and order at most $n$ from the lists $\cG(n,a,b)$ and the lists $\cS(G,a,b)$, for each $(a,b)\in \Sigma_k$ and $G\in\cG(n,a,b)$.

Finally, step \ref{ls:general_approach_iii} will take care of all redundancies in the resulting set of $k$-valent Cayley graphs $\{\cay(G,S):(a,b)\in \Sigma_k,G\in\cG(n,a,b),S\in\cS(G,a,b)\}$. Such redundancies are caused by at least one of the following: having two representatives of the same $\aut(G)$-orbit on $\cS(G,a,b)$, or having Cayley sets of two non-isomorphic groups or non-equivalent Cayley sets of the same group that yield isomorphic Cayley graphs.

\section{Constructing a list of sub-$(a,b)$-valent groups}\label{s:subvalent_groups}

Let $k$ be a positve integer and $(a,b)\in \Sigma_k$ througout this section. The methods we use to enumerate sub-$(a,b)$-valent groups in this paper are largely the same as those presented in \cite{PSV_2013a,PSV_2016}, differing only in their application and the inclusion of some basic theoretical observations.

We will introduce two methods of constructing a list $\cG(n,a,b)$ containing the sub-$(a,b)$-valent groups of order at most $n$: filtering through known libraries of groups (Subsection \ref{ss:lib_groups}) and group extensions (Subsection \ref{ss:extension_methods}). 
For several values of $(a,b)$ and $n$ we have used both of these methods, and find that both result in the same sub-$(a,b)$-valent groups (up to isomorphism). This gives empirical evidence of the correctness of the computations. It is also possible to combine both methods to capitalise on their respective advantages at different stages of the computation. 

Reducing the size of $\cG(n,a,b)$ will also reduce the number of computations carried out in steps \ref{ls:general_approach_ii} and \ref{ls:general_approach_iii}. This can result in an increase in efficiency and allow for computations with larger values of $k$ and $n$. However, reducing $\cG(n,a,b)$ fully can also be computationally difficult. Subsection \ref{ss:min_normal} gives a simple reduction method based on taking quotients.

~
\subsection{Quotients by minimal normal subgroups}\label{ss:min_normal}

We will often construct a list of potential sub-$(a,b)$-valent groups of order $n$, and would like to recognise that a group in the list is not sub-$(a,b)$-valent as efficiently as possible. A useful method to do so comes from considering normal quotients.

Suppose we are building a list $\cG(n,a,b)$ containing all sub-$(a,b)$-valent groups of orders at most $n$, and we have already constructed a partial sublist $\cG'$ containing all sub-$(a,b)$-valent groups of orders dividing $n$ nontrivially. If a group $G$ of order $n$ is sub-$(a,b)$-valent, then each quotient of $G$ by a nontrivial normal subgroup is isomorphic to a sub-$(a,b)$-valent group in $\cG'$. It follows that if $G$ is a group with a quotient that is not sub-$(a,b)$-valent, it does not need to be added to our partial sublist of sub-$(a,b)$-valent groups. 

This process can be used to filter out large numbers of groups that are not sub-$(a,b)$-valent, leaving us with the groups for which every quotient by a nontrivial normal subgroup is sub-$(a,b)$-valent. The work of Volta and Lucchini \cite[Theorem 1.4]{LV_1998} suggests that the groups which are not sub-$k$-valent but have only sub-$k$-valent quotients are somewhat special. 

Although isomorphism calculations are computationally costly in general, we use precomputed isomorphism invariants to improve efficiency. Such functionality and further improvements are already implemented in GAP and Magma, with substantial optimisation available for the groups of order at most 2000. One can also use canonicalisation for some classes of groups (e.g., $p$-groups \cite{BEO_2002,ANUPQ}) to avoid pairwise isomorphism calculations, although we have not found this necessary for our calculations so far.

To keep the number of isomorphism calculations as low as possible, we would like to reduce the number of quotients we need to consider. By Lemma \ref{lem:min_normal}, we only need to check quotients by minimal normal subgroups to check the condition for all nontrivial quotients. Considering only minimal normal subgroups has the added benefits of reduced computational time and memory usage.
\begin{lem}\label{lem:min_normal}
	Let $G$ be a group. If every quotient of $G$ by a minimal normal subgroup of $G$ is sub-$(a,b)$-valent, then every quotient of $G$ by a nontrivial normal subgroup is sub-$(a,b)$-valent.
\end{lem}
\begin{proof}
	Let $S$ be a sub-$(a,b)$-valent Cayley set of $G$ and $N\trianglelefteq G$ be nontrivial. Then there exists a minimal normal subgroup $M\trianglelefteq G$ such that $M\leq N$. As $G/M$ is sub-$(a,b)$-valent and $G/N\cong (G/M) / (N/M)$ by the isomorphism theorems, it follows from Lemma \ref{lem:quotient_closed} that $G/N$ is a sub-$(a,b)$-valent group.
\end{proof} 

Similarly to Lemma \ref{lem:quotient_closed}, we note that Lemma \ref{lem:min_normal} and all arguments in this subsection can be applied to the quotient-closed class of sub-$k$-valent groups.

\subsection{Using libraries of small groups}\label{ss:lib_groups}

The most accessible method of constructing the list $\cG(n,a,b)$ is to filter through a library of all groups of order at most $n$. In particular, we use the complete lists of groups of order at most 2000, except for order 1024, that are available in the GAP \cite{GAP,SmallGrp_2024} and Magma \cite{MAGMA_1997} computational algebra systems. These systems give easy access to every group of such orders, along with several precomputed attributes for each group. In fact, all groups of order 1024 that have rank at most 4 are also available through the GRPS1024 package for GAP \cite{B_2022}.

In this method, we iterate through the possible orders of a sub-$(a,b)$-valent group in increasing order. Suppose we are now considering the groups of order $n$. First, we filter out any group of rank greater than $k$. This is often very cost efficient, as rank is precomputed for most of the groups in the libraries. Then we check the quotients of each group as described in Subsection \ref{ss:min_normal}. After applying this to all possible orders, the remaining groups are added to the list $\cG(n,a,b)$.

This approach does not involve complex group construction methods. The precomputed attributes and carefully chosen presentations of each group in the library allow for immediate access to optimised isomorphism and normal subgroup functionalities. However, filtering the groups which are not sub-$(a,b)$-valent can still take a significant amount of computational time, as the proportion of groups of order $n$ that are sub-$(a,b)$-valent is very small. Therefore, this methods is found to be impractical for orders with a large number of groups (e.g, orders 1024 and 1536). This method is also limited by the range of orders of groups available in the libraries of groups. To extend our enumerations to the groups of several problematic orders and beyond order 2000, we need a method to construct sub-$(a,b)$-valent groups efficiently.

\subsection{Group extensions}\label{ss:extension_methods}

There are several known methods for enumerating finite groups of a given order \cite{BEO_2002,EHH_2017}. The efficiency of each method is affected by the properties of the groups one is enumerating.

Most of these methods involve one of two possible approaches:
\begin{itemize}
    \item Computing finite quotients of an infinite finitely presented group distinguished by a universality property for the class of groups that are being enumerated.
    \item Extending groups from a complete list of possible quotients, assuming some specific structure of the assumed normal subgroup.
\end{itemize} 
The first approach was successfully used, for example, in the enumeration of all (automorphism groups of) cubic arc-transitive graphs~\cite{ConPot25} of order at most $10{,}000$, or in the enumeration of group actions on compact Riemann surfaces of bounded genus~\cite{MarstonHomepage}.
This approach relies heavily on an algorithm of Holt and Firth~\cite{HoltFirth}, which computes all
normal subgroups of a finitely presented group up to a given index. Efficiency of the current implementation in Magma (which has an in-built limit of $500{,}000$ on the index of a normal subgroup) depends both on the number of the generators of the infinite group as well as on the complexity of the presentation, and is often hard to predict in advance. In the case of sub-$(a,b)$-groups, one could use this approach on the universal group $\langle x_1, \ldots, x_a, y_1, \ldots, y_b \mid x_1^2=\ldots=x_a^2=1\rangle$. However, it turns out that this approach is inferior to the one presented below, even in the case when $a+b = 3$.

The second approach builds the groups from below, starting with an appropriate set of basic groups (see Lemma~\ref{lem:inductive_radical}) and then extending these recursively by elementary-abelian groups. 
Several algorithms to enumerate extensions of specified structure are implemented in GAP \cite{ANUPQ,GrpConst}. However, in the case of low rank groups we find in significantly more efficient to use bespoke implementations in Magma that follow the approaches in \cite{PSV_2013a,PSV_2016}.
Let us explain this in more detail. 

Let $G$ be a group, $N\unlhd G$ and $Q\cong G/N$. Then $G$ is an \emph{extension} of $Q$ by $N$. Furthermore, the extension is \emph{direct} if $N$ is a minimal normal subgroup of $G$, and \emph{elementary abelian} if $N$ is an elementary abelian group. Generating extensions of a given group is often computationally difficult, but the process can be made easier if we restrict our search space to direct elementary abelian extensions.

Moreover, since we are only interested in the extensions of orders that do not exceed some prescribed bound, say $M$, we only need to consider extensions by elementary abelian groups $\ZZ_p^d$ such that $d\le \log_p(M/n)$, where $n$ is the order of the group that is being extended. In practice, the vast majority of extensions we need to perform are by cyclic groups $\ZZ_p$ and rank-2 groups $\ZZ_p^2$.

In the following we repeat arguments found in \cite{PSV_2016}, but in the application to $(a,b)$-valent groups. We remind the reader that the \emph{soluble radical} of a group $G$ is the largest soluble normal subgroup of $G$.

\begin{lem}\label{lem:inductive_radical}
    Let $G$ be an $(a,b)$-valent group. Then $G$ has trivial soluble radical, or is a direct elementary abelian extension of a sub-$(a,b)$-valent group. 
\end{lem}
\begin{proof}
    Suppose $G$ has non-trivial soluble radical $S$, and let $N\unlhd G$ be a minimal normal subgroup of $G$ contained in $S$. Then $N$ is a minimal normal subgroup in $S$, and hence must be elementary abelian. The result follows from Lemma \ref{lem:quotient_closed}.
\end{proof}

\paragraph{\textbf{The inductive step.}} Lemma \ref{lem:inductive_radical} allows one to inductively construct $(a,b)$-valent groups. For this, we use the $(a, b)$-valent groups of order at most $n$ with trivial soluble radical as the base case. For each group $G$ with trivial soluble radical, $\mathrm{soc}(G)$ is isomorphic to a direct product of non-abelian simple groups. Moreover, $G$ acts faithfully on $\mathrm{soc}(G)$ by conjugation and thus $G$ embeds into $\aut(\mathrm{soc}(G))$. Such groups are sparse amongst the groups up to a given order. Therefore, the database of small simple groups in Magma can be used to construct all groups of order at most $n$ with trivial soluble radical, for small $n$. In particular, one finds there are 23 groups with trivial soluble radical of order at most 6000.

In the inductive step, we compute direct elementary abelian extensions of a given group. Let $Q$ be a group and $N$ be an elementary abelian $p$-group for prime $p$. To compute all direct elementary abelian extensions of $Q$ by $N$, it is enough to compute all irreducible $\ZZ_pQ$-module structures $M$ of $N$, and then compute the cohomology group $H^2(Q,M)$ for each $M$. It is known that each element of $H^2(Q, M)$ then gives rise to a direct extension of $Q$ by $N$. Furthermore, each direct elementary abelian extension of $Q$ by $N$ can be obtained in this manner. 

\paragraph{\textbf{Computing cohomology groups.}}
Algorithms for the computation of all irreducible modules and cohomology groups are implemented in Magma, alongside functions for the direct computation of certain classes of group extensions. 
However, the implementations of the algorithms for computing  irreducible $\ZZ_pG$-modules available in Magma were not sufficiently efficient for our purposes. This is mainly due to the fact that the Magma implementations always attempt to compute {\bf all} irreducible $\ZZ_pG$-modules, with no optimisation when asking for only modules up to a given $\ZZ_p$-dimension. However, as we already mentioned in the discussion preceding Lemma~\ref{lem:inductive_radical}, in the majority of cases we only need to compute extensions by elementary abelian groups of bounded (typically very small) rank. This allowed us to devise and implement an alternative approach, where we first precompute all irreducible subgroups of relevant general linear groups $\GL{d}{\ZZ_p}$ (for $d$ bounded), and then for each group $G$ that needs to be extended, find all epimorphisms from $G$ to the precomputed irreducible subgroups of $\GL{d}{\ZZ_p}$, for all primes $p$ and positive integers $d$ such that the order $p^d|G|$ of the extended group does not exceed the limit on the order of groups we want to find (say, $5{,}000$ in the case of $3$-valent groups). When $d$ is small, this simple shortcut significantly speeds up the process of finding all irreducible $d$-dimensional $\ZZ_pG$-modules.

\paragraph{\textbf{Further optimisations.}}
During the process of generating a list of groups as group extensions, we occasionally use careful ordering considerations to optimise computations. Suppose we are generating a list $\cG$ of sub-$(a,b)$-valent groups of order $n$ by computing the direct extensions of a complete list of possible quotients $\cQ$ by a set of minimal normal subgroups $\cN$. Some groups will be the result of many nonequivalent extensions, so they will be generated several times. To filter out copies coming from different groups in $\cQ$, one can use an approach similar to Subsection \ref{ss:min_normal}. 

Assume we construct a group $G\in\cG$ as a direct extension of a group $Q\in\cQ$ by a minimal normal subgroup $N\in\cN$. Then we compute all minimal normal subgroups of $G$, and check that all corresponding quotients are sub-$(a,b)$-valent (Subsection \ref{ss:min_normal}). When doing so, we can also check if a minimal normal subgroup $N'\leq G$ is such that $N'$ is isomorphic to an element of $\cN$ and $Q'=G/N'$ is isomorphic to an element of $\cQ$. If so, and all extensions of $Q'$ by $N'$ have already been computed, we can discard $G$. In this way each $G$ is constructed only once, as the extension of the first pair $(Q,N)\in\cQ\times\cN$ we consider such that $G/N\cong Q$ and $N\unlhd G$ is a minimal normal subgroup. 

One bottleneck in this computation is when testing for isomorphisms of the minimal normal subgroups and the quotients with elements of $\cN$ and $\cQ$, respectively. However, we can often argue theoretically to find a set $\cN$ consisting of a small number of groups. For example, $\cN=\{\ZZ_p\}$ suffices for any $p$-group as all minimal normal subgroups of a $p$-group are isomorphic to $\ZZ_p$. This characterisation also allow one to omit all isomorphism tests between minimal normal subgroups and $\ZZ_p$. The implementation of isomorphism tests for quotients by minimal normal subgroups are handled in the same way as described in Subsection \ref{ss:min_normal}.

Extension methods afford the opportunity to generate all sub-$(a,b)$-valent groups of large orders. The proportion of groups of order $n$ that are sub-$(a,b)$-valent is small for large $n$. As sub-$(a,b)$-valent groups are closed under taking quotients, we only need to apply extension methods to a small proportion of groups in our computations. This is particularly useful when considering groups of order a large power of 2, as such groups account for the overwhelming majority of groups of relatively small order.

\section{Constructing a list of $(a,b)$-valent Cayley sets}\label{s:cayley_sets}

Let $k$ be a positve integer and $(a,b)\in \Sigma_k$ througout this section. Given a group $G$, we introduce two methods to construct a list $\cS(G,a,b)$ of $(a,b)$-valent Cayley sets of $G$ having at least one representative of each $\aut(G)$-orbit. 

Let $G^\pm=\{\{x,x^{-1}\}:x\in G\setminus \tgroup\}$. Note that this is a partition of $G\setminus\tgroup$ into sets of size 1 or 2, and $\aut(G)$ acts on $G^\pm$ naturally. Furthermore, given a Cayley set $S$ of $G$, we can identify $S$ with a unique subset of $G^\pm$. For the remainder of this section, we will use $G^\pm$ as the base set for constructing Cayley sets of $G$.

\subsection{Orderly generation}\label{ss:orderly_generation}

One method to construct the list $\cS(G,a,b)$ is to use a simple orderly generation algorithm \cite{R_1978,F_1978}, which has been used in many similar applications (e.g., \cite{HR_2020,HRT_2022}). The implementation of this method relies on the ability to compute the smallest lexicgraphic image of a set under the action of a permutation group. In GAP, the Grape package \cite{GRAPE} function \texttt{SmallestImageSet} and the images package \cite{images} provide such functionality. 
\begin{algorithm}[]
	\caption{Orderly generation of $(a,b)$-valent Cayley sets of the group $G$}\label{alg:orderly_generate}
	\begin{algorithmic}
		\State $G^\pm\gets\{\{x,x^{-1}\}:x\in G\setminus \tgroup\}$
            \State $I\gets \{y\in G^\pm:|y|=1\}$, $N\gets G^\pm\setminus I$
          \State $\ll\medskip\gets$ strict total ordering of $G^\pm$ such that $\forall y\in I,\forall x\in N, y\ll x $
		\Function{OrderlyGenerate}{T} \Comment{define the recursive function}
            \If{$|\bigcup T|=k$}\Comment{where $\bigcup T := \bigcup_{x\in T}x$}
	   	  \If{$\gen{\bigcup T}=G$}  
	   	    \State \Record~$T$ \Comment{or add to $\cS(G,a,b)$}
    		  \EndIf
                \State \Return
            \EndIf
            \State $t\gets|\bigcup T|$
            \LComment{We extend $T$ in an `orderly' manner, so we can filter out subsets that cannot be contained in an orderly-generated $(a,b)$-valent Cayley set.}
            \If{$t<a$} \Comment{We add involutions first}
               \State $E\gets\{y\in I:\max_{\ll}(T)\ll y,|\{x\in I:y\ll x\}|\geq a-t-1\}$
            \Else
               \State $E\gets\{y\in N:\max_{\ll}(T)\ll y,|\{x\in N:y\ll x\}|\geq (k-t-2)/2\}$
            \EndIf
	      \For{$y\in E$}
	        \If{\Call{IsSmallestImage}{$\aut(G),T\cup\{y\}}$} \Comment{lex$_\ll$ ordering}
		       \State \Call{OrderlyGenerate}{$T\cup\{y\}$}
		    \EndIf   
		  \EndFor
		\EndFunction
		\Call{OrderlyGenerate}{$\emptyset$} \Comment{start the full recursive process} 
	\end{algorithmic}
\end{algorithm}

Pseudocode for this method is presented in Algorithm \ref{alg:orderly_generate}. This method can also be adapted to generate Cayley sets with valencies in any subset of $\Sigma_k$. Furthermore, the generation of $(k,0)$-valent Cayley sets can be done by restricting the base set $G^\pm$ to involutions in $G$. 

This method records exactly one $(a,b)$-valent Cayley set from each $\aut(G)$-orbit, constructing a list $\cS(G,a,b)$ of minimum possible size. The algorithm traverses a search tree in a depth-first manner, resulting in reduced memory requirements and easy parallelisation. However, computing the smallest image of a subset under the action of a permutation group is computationally hard, and many such computations may be needed to traverse the full search tree.

\subsection{Ordered generation}\label{ss:ordered_generation}

We also use method to construct the list $\cS(G,a,b)$ that avoids the use of a smallest image function. In this method we construct an ordered list $L$ recursively, and reduce the possible extensions of $L$ by using equivalence under the (pointwise) stabiliser of $L$ in $\aut(G)$.

\begin{algorithm}[]
	\caption{Ordered generation of $(a,b)$-valent Cayley sets of the group $G$}\label{alg:ordered_generate}
	\begin{algorithmic}
		\State $G^\pm\gets\{\{x,x^{-1}\}:x\in G\setminus \tgroup\}$
            \State $I\gets \{y\in G^\pm:|y|=1\}$, $N\gets G^\pm\setminus I$
		\State $\ll\medskip\gets$ strict total ordering of $G^\pm$
		\Function{OrderedGenerate}{L} \Comment{define the recursive function}
            \If {$|\bigcup L|=k$} \Comment{where $\bigcup L := \bigcup_{x\in L}x$}
    		  \If{$\gen{\bigcup L}=G$}    
	            \State \Record~$L$ \Comment{or add to $\cS(G,a,b)$}
                \EndIf
                \State \Return
            \EndIf
            \If{$|\bigcup L|<a$}\Comment{add $a$ involutions first}
               \State $E\gets I\setminus L$
            \Else
               \State $E\gets N\setminus L$	 
            \EndIf   
	   	\State $A\gets \aut(G)_{(L)}$
	      \State $R\gets$\Call{SmallestRepresentative}{$\Orbits{A}{E}$}
		\For{$y\in R$}
    		\State \Call{OrderedGenerate}{$L+[y]$}\Comment{`$+$' here means concatenation}
		\EndFor
		\EndFunction
		\Call{OrderedGenerate}{$[~]$} \Comment{Start the full recursive process} 
	\end{algorithmic}
\end{algorithm}

Pseudocode for this method is presented in Algorithm \ref{alg:ordered_generate}. In particular, the resulting set $\cS(G,a,b)$ will not necessarily be reduced up to the action of $\aut(G)$ on Cayley sets. However, $\cS(G,a,b)$ will contain at least one representative for each $\aut(G)$-orbit of Cayley sets. For some orders $n$ for which the number of groups of order $n$ is large, this method has been used in preference to orderly generation, due to the observed computational times. 

Only a few simple modifications to Algorithm \ref{alg:ordered_generate} are needed to implement a canonical augmentation method \cite{M_1998}. Such an algorithm may combine the benefits of reduced use of canonicalisation functions and the complete reduction of the resulting set $\cS(G,a,b)$ up to the action of $\aut(G)$. We have not yet implemented such an algorithm, but we plan to test and analyse this method in the future.

\section{Constructing the Cayley graphs}\label{s:extract_cayley}

Suppose we have computed lists $\cG(n,a,b)$ and $\cS(G,a,b)$ for each $(a,b)\in \Sigma_k$ and $G\in\cG(n,a,b)$ as in steps \ref{ls:general_approach_i} and \ref{ls:general_approach_ii}. We then compute the Cayley graphs $\cay(G,S)$ for each $(a,b)\in \Sigma_k$, $G\in\cG(n,a,b)$ and $S\in \cS(G,a,b)$, and reduce the resulting list of graphs by isomorphism. This reduction can be done using nauty \cite{nauty_2024}. In particular, we use the sparse internal representation option of the \texttt{shortg} utility to optimise calculations for Cayley graphs of low valency. We store the resulting graphs in sparse6 format \cite{nauty_2024}. 

\section{Statistical observations}\label{s:stats}

The methods presented in Sections \ref{s:subvalent_groups}, \ref{s:cayley_sets} and \ref{s:extract_cayley} have been applied to the generation of complete irredundant lists of cubic Cayley graphs of order at most 5000 \cite{CUBIC_2025}, and quartic Cayley graphs of order at most 1025 \cite{QUARTIC_2025}. In total we find 1939864 cubic and 11361600 quartic Cayley graphs, the distribution of which is illustrated in Figure \ref{fig:count_c}.

\begin{figure}[h]
    \centering
       \begin{minipage}{0.5\textwidth}
        \centering
        \includegraphics[width=\textwidth]{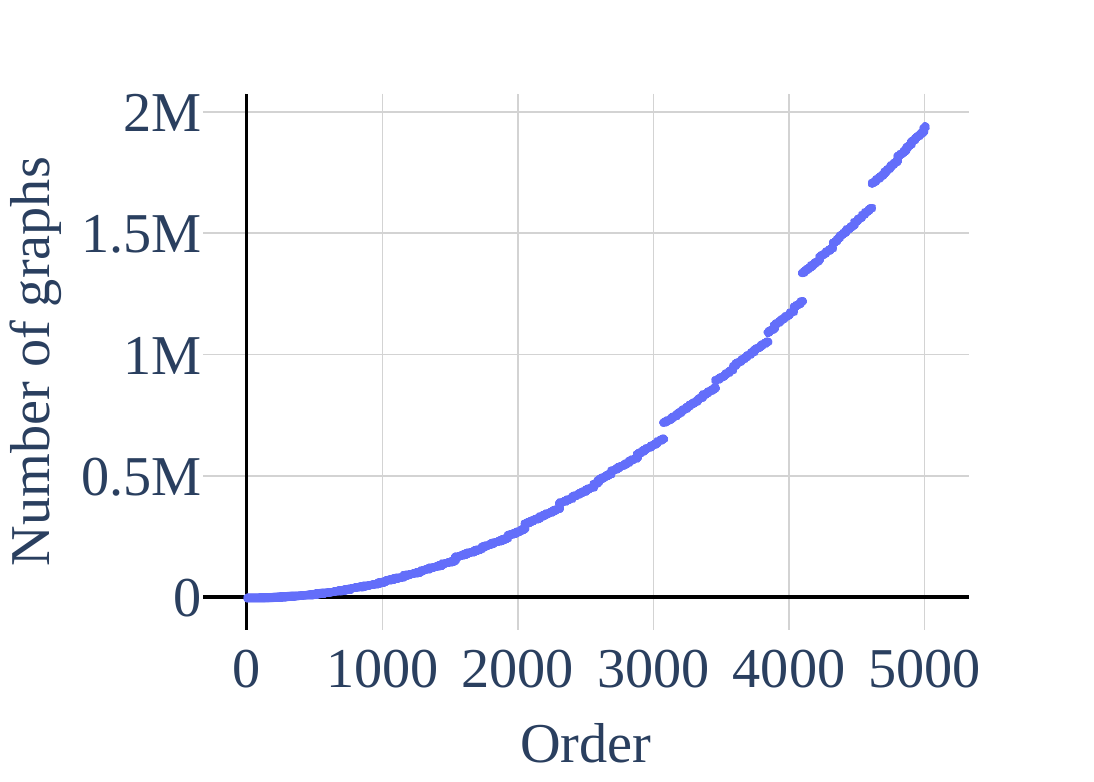} 
    \end{minipage}\hfill
    \begin{minipage}{0.5\textwidth}
        \centering
        \includegraphics[width=\textwidth]{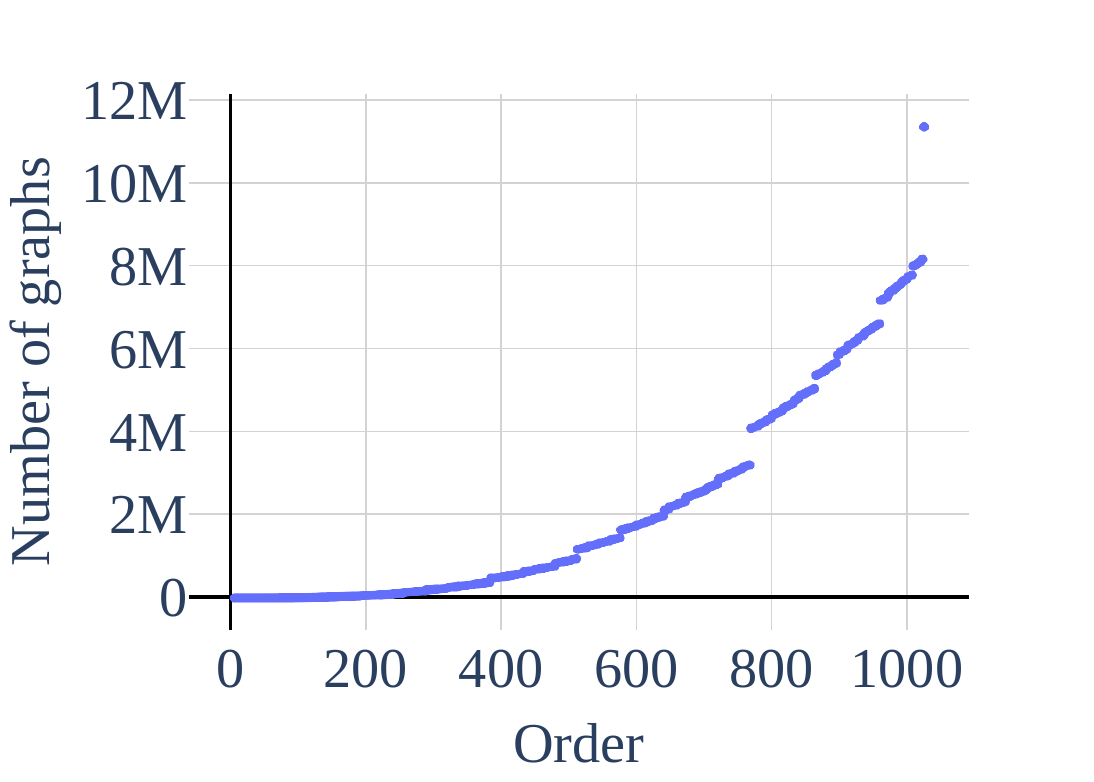} 
    \end{minipage}
    \caption{Cumulative number of cubic (L) and quartic (R) Cayley graphs of each order.}
    \label{fig:count_c}
\end{figure}

In this section we analyse these collections of graphs, and propose further investigations into the theory behind some observable patterns. For each figure caption, left and right scatter plots are referred to by (L) and (R), respectively.

\subsection{GRRs, arc-transitivity and vertex stabilisers}\label{ss:aut_data}

A Cayley graph $\Gamma$ of a group $G$ is said to be a \emph{graphical regular representation (GRR)} if $\aut(\Gamma)\cong G$. This is equivalent to the condition that the stabiliser of a vertex in $\aut(\Gamma)$ is trivial. 

The class of GRRs has seen much attention since their conception, with a particular interest in their frequency. Recently, Xia and Zheng \cite{XS_2023} proved that the proportion of Cayley graphs with at most $n$ vertices that are GRRs tends to 1 as $n\to\infty$ (note that they do not assume that Cayley graphs are connected). Figure \ref{fig:grr_c} suggests this is also true if we restrict to cubic or quartic Cayley graphs. In general, we can ask if this is true for $k$-valent Cayley graphs of any fixed valence $k$.
\begin{problem}
    Let $k\geq 3$ be an integer. Does the proportion of $k$-valent Cayley graphs with at most $n$ vertices that are GRRs tend to 1 as $n\to\infty$?
\end{problem}

\begin{figure}[h]
    \centering
       \begin{minipage}{0.5\textwidth}
        \centering
            \includegraphics[width=\textwidth]{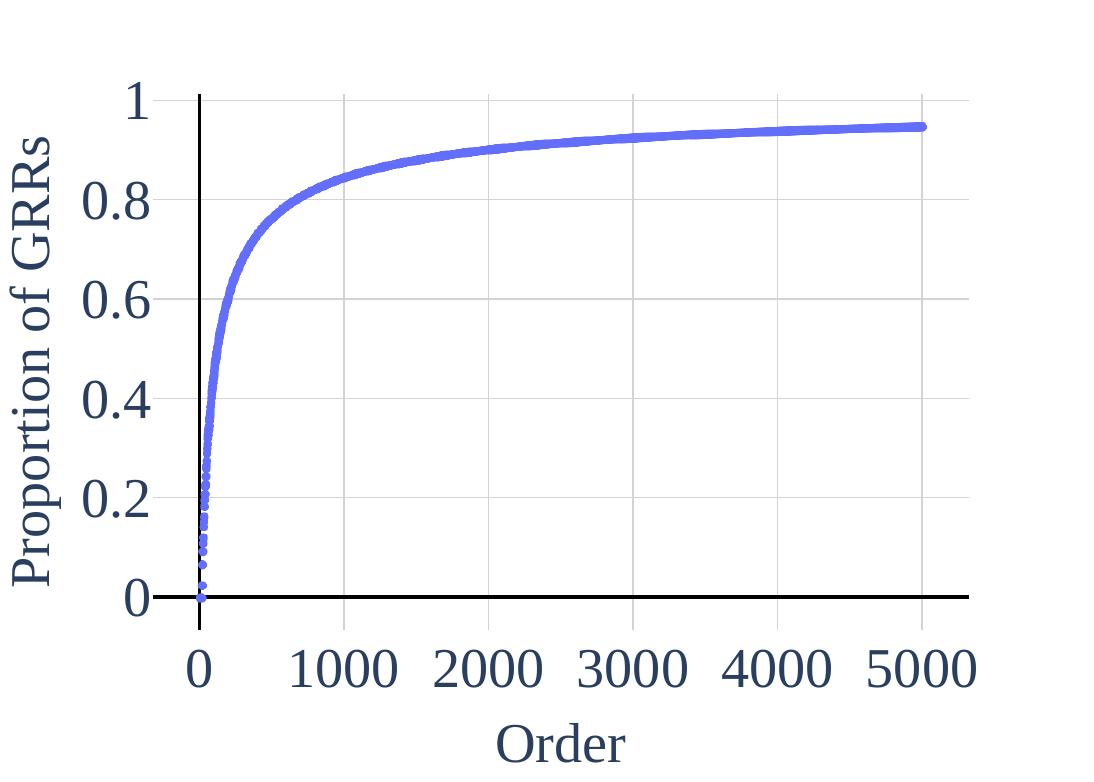}
    \end{minipage}\hfill
    \begin{minipage}{0.5\textwidth}
        \centering
         \includegraphics[width=\textwidth]{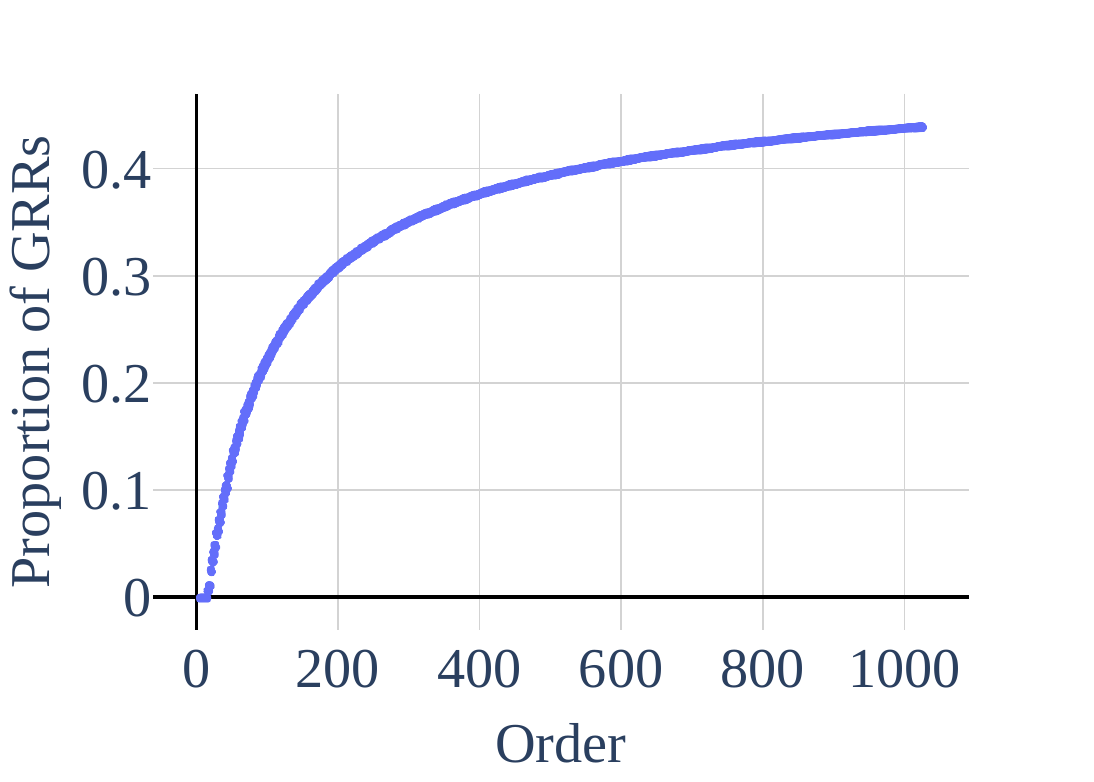} 
    \end{minipage}
    \caption{Cumulative proportion of cubic (L) and quartic (R) Cayley graphs that are GRRs.}
    \label{fig:grr_c}
\end{figure}

However, we observe that the rate at which the proportion of GRRs increases is significantly smaller for quartic Cayley graphs than for cubic Cayley graphs. This can be explained by the existence of quartic Cayley graphs with an odd number of vertices, and the difference between the odd and even order cases (see Figure \ref{fig:grr_quar_comp}).

\begin{figure}[h]
    \centering
         \includegraphics[width=\textwidth]{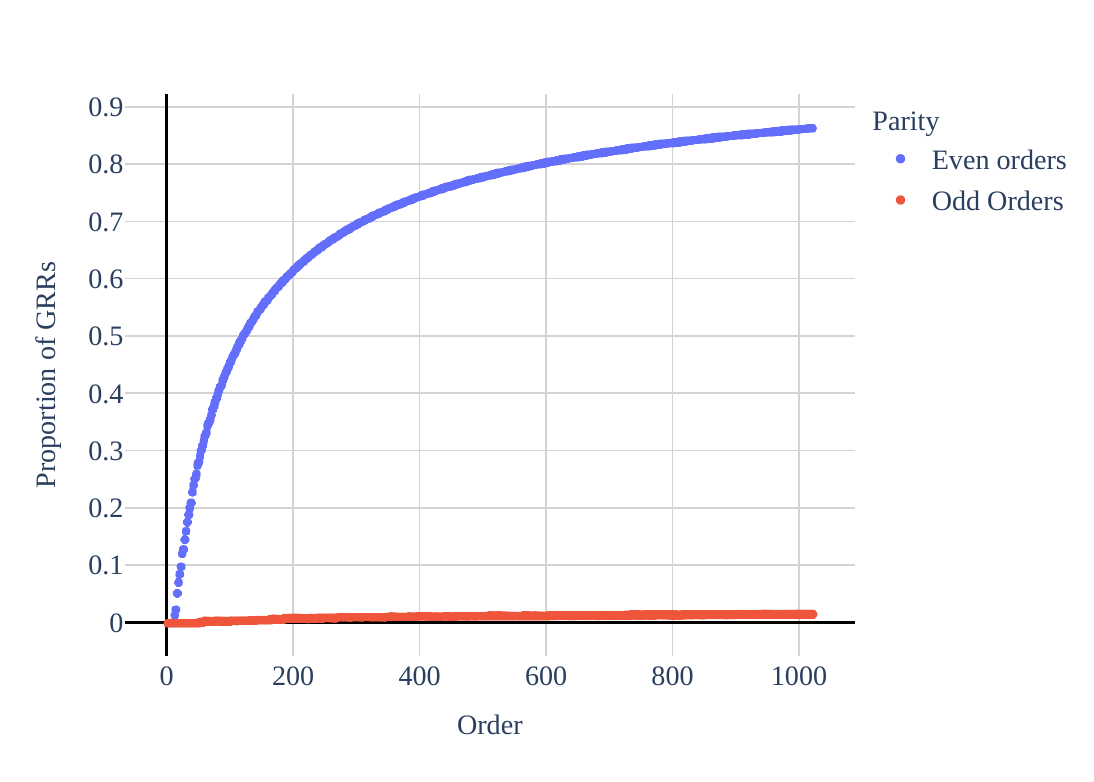} 
    \caption{Cumulative proportion of quartic Cayley graphs that are GRRs for odd and even orders.}
    \label{fig:grr_quar_comp}
\end{figure}

A graph $\Gamma$ is \emph{arc-transitive} if $\aut(\Gamma)$ acts transitively on the arcs of $\Gamma$. If $\Gamma$ is an arc-transitive Cayley graph, then it cannot be a GRR. In some sense, being a GRR and being arc-transitive are antipodal properties in the class of Cayley graphs.

Figure \ref{fig:at_c} illustrates the proportion of cubic and quartic Cayley graphs which are arc-transitive. As expected from the proportions of GRRs, the proportions of arc-transitive Cayley graphs tend to 0 as $n\to\infty$. We may therefore ask a similar question as the above.

\begin{problem}
    Let $k\geq 3$ be an integer. Does the proportion of $k$-valent Cayley graphs with at most $n$ vertices that are arc-transitive tend to 0 as $n\to\infty$?
\end{problem}

\begin{figure}[h]
    \centering
       \begin{minipage}{0.5\textwidth}
        \centering
          \includegraphics[width=\textwidth]{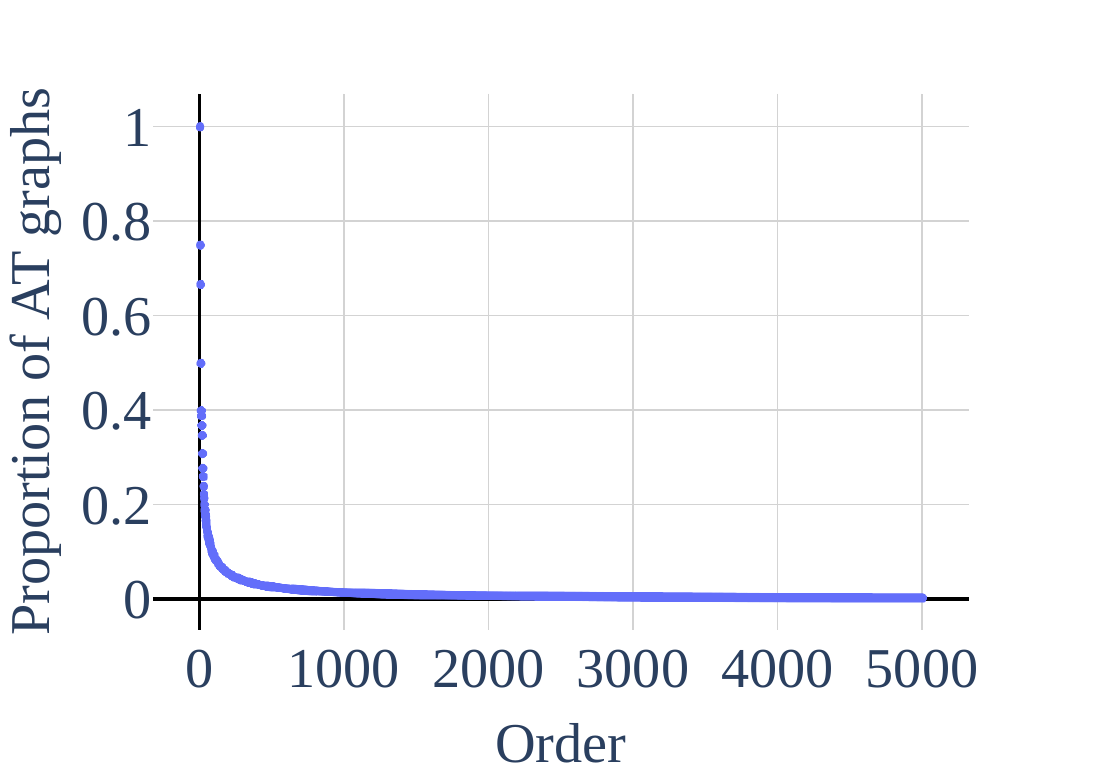} 
    \end{minipage}\hfill
   \begin{minipage}{0.5\textwidth}
        \centering
         \includegraphics[width=\textwidth]{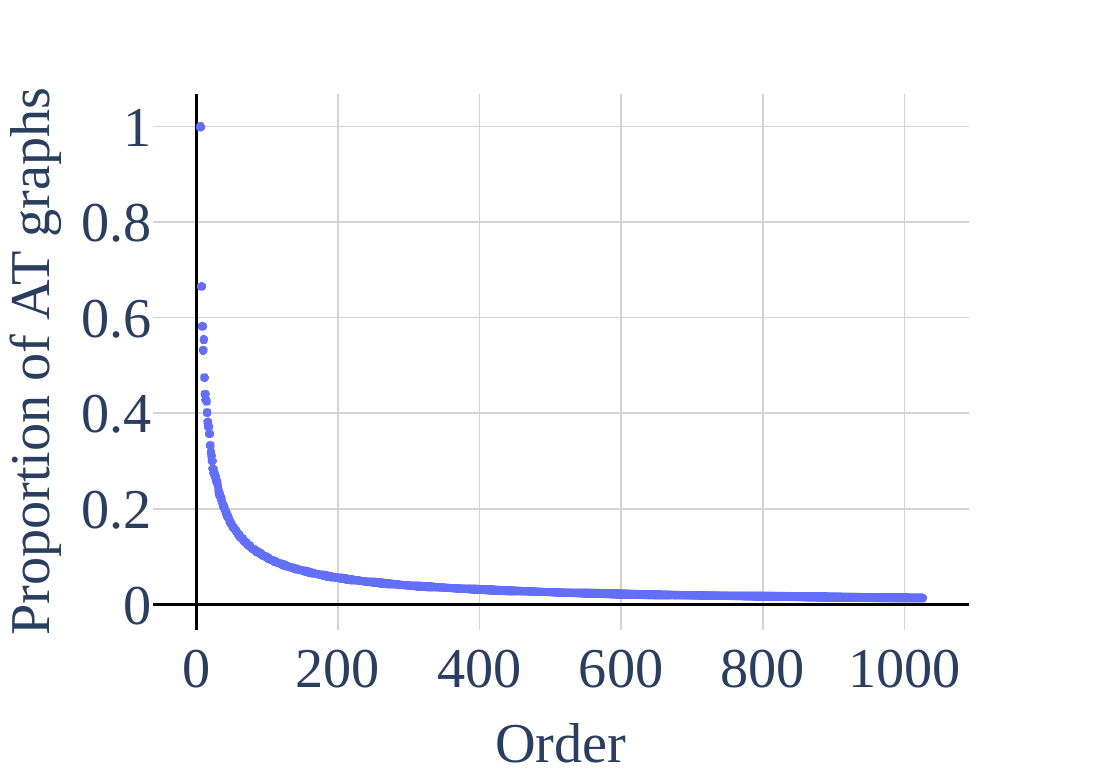} 
    \end{minipage}
    \caption{Cumulative proportion of cubic (L) and quartic (R) Cayley graphs that are arc-transitive.}
    \label{fig:at_c}
\end{figure}

Now let us consider only the cubic Cayley graphs that are neither a GRR or arc-transitive. Such a graph $\Gamma$ with vertex $v$ must be such that $|\aut(\Gamma)_v|=2^t$ for some $t\geq 1$. Let $s\geq 1$ be an integer and $\cA(n,s)$ be the set of cubic Cayley graphs with at most $n$ vertices and vertex-stabiliser of order $2^t$, for any $t\geq s$. Further, let 
$$a(n,s):=\frac{|\{\Gamma\in \cA(n,s):|\aut(\Gamma)_v|=2^s\}|}{|\cA(n,s)|}.$$
One can think of $a(n,s)$ as a measure of how often the stabiliser of graphs in $\cA(n,s)$ is as small as possible. Figure \ref{fig:stab} suggests that $a(n,s)$ grows steadily as $n\to\infty$. This prompts the following question. 

\begin{problem}
 Let $s\geq 1$ be an integer. Does the proportion $a(n,s)$ tend to 1 as $n\to\infty$?
\end{problem}

\begin{figure}[h]
    \centering
         \includegraphics[width=\textwidth]{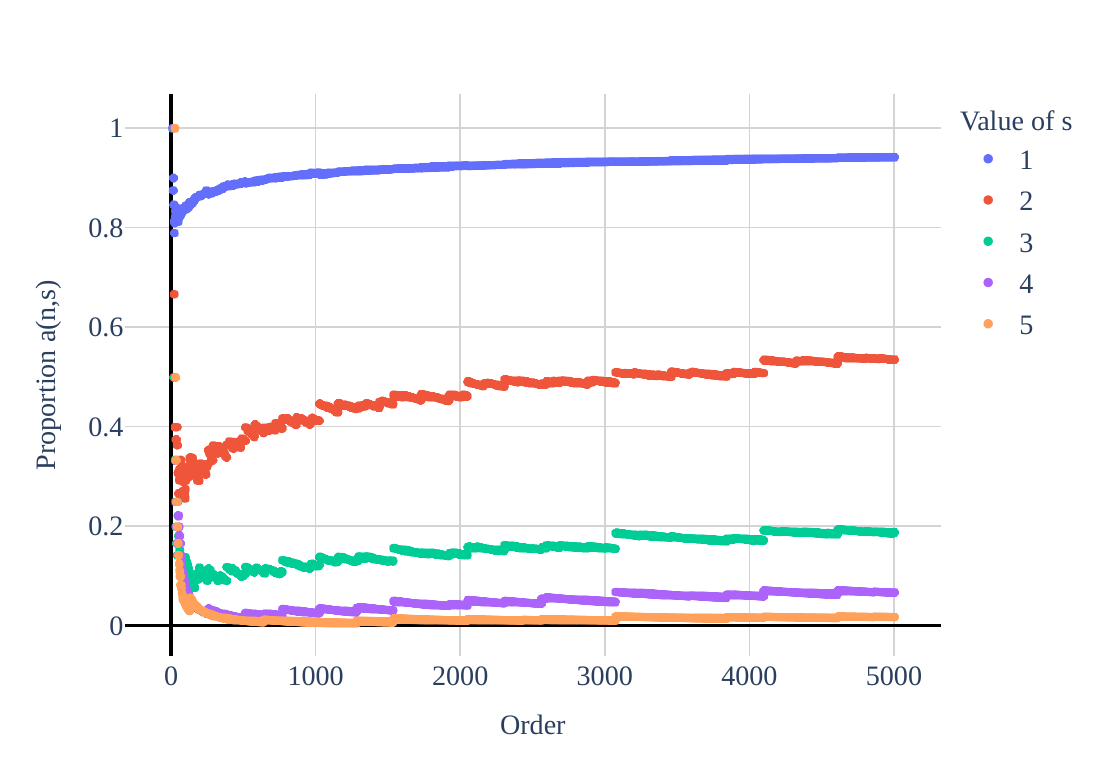} 
    \caption{The proportion $a(n,s)$ for $s\in\{1,2,3,4,5\}$.}
    \label{fig:stab}
\end{figure}

\subsection{Girth}\label{ss:girth_data}

The \emph{girth} of a graph $\Gamma$ is the length of the shortest cycle in $\Gamma$. Given integers  $k\geq 2,g\geq 3$, a \emph{$(k,g)$-cage} is a $k$-regular graph with girth $g$ having the smallest possible number of vertices. 

The study of $(k,g)$-cages was first presented by Tutte \cite{T_1947}, and many variations of these graphs can be found in the literature \cite{EJ_2011}. In particular, we are interested in \emph{Cayley $(k,g)$-cages}: the $k$-valent Cayley graphs with girth $g$ having the smallest possible number of vertices. We denote the size of a Cayley $(k,g)$-cage by $n_c(k,g)$. Table \ref{tab:cages} gives the values for $n_c(k,g)$ that can be extracted from our collections of cubic and quartic Cayley graphs.

\begin{table}[h]
    
    \begin{tabular}{|c||c|c|c|c|c|c|c|c|}
    \hline
      $g$ & 3 & 4 & 5 & 6 & 7 & 8 & 9 & 10 \\
      \hline
       $n_c(3,g)$ & 4 & 6 & 50 & 14 & 30 & 42 & 60 & 96\\
       \hline
    \end{tabular}       

\vspace{1em}
    
    \begin{tabular}{|c||c|c|c|c|c|c|c|c|}   
       \hline
       $g$ & 11 & 12 & 13 & 14 & 15 & 16 & 17 & 18 \\
       \hline
       $n_c(3,g)$ & 192 & 162 & 272 & 406 & 864 & 1008 & 3024 & 2640 \\
       \hline
    \end{tabular}
    
    \vspace{1em}
    
    \begin{tabular}{|c||c|c|c|c|c|c|c|c|}
    \hline
      $g$ & 3 & 4 & 5 & 6 & 7 & 8 & 9 & 10 \\
      \hline
       $n_c(4,g)$ & 5 & 8 & 24 & 26 & 72 & 96 & 320 & 410\\
       \hline
    \end{tabular}    
    
    \caption{Values of $n_c(3,g)$ and $n_c(4,g)$ for small $g$.}
    \label{tab:cages}
\end{table}

Given integers $k\geq 2,g\geq 3$, let $O(k,g)$ denote the possible orders of a $k$-valent Cayley graph with girth $g$. Note that $n_c(k,g)=\min O(k,g)$. We can also look at the natural density of $O(k,g)$, which is defined as $$D(k,g)=\lim_{n\to\infty} \frac{|O(k,g)\cap \{1,\dots,n\}|}{n}.$$ Figure \ref{fig:girth_density} illustrates the first values of the sequence within this limit, for $k\in\{3,4\}$ and $g\in\{3,\dots,10\}$.

\begin{figure}[h]
    \centering
       \begin{minipage}{0.5\textwidth}
        \centering
        \includegraphics[width=\textwidth]{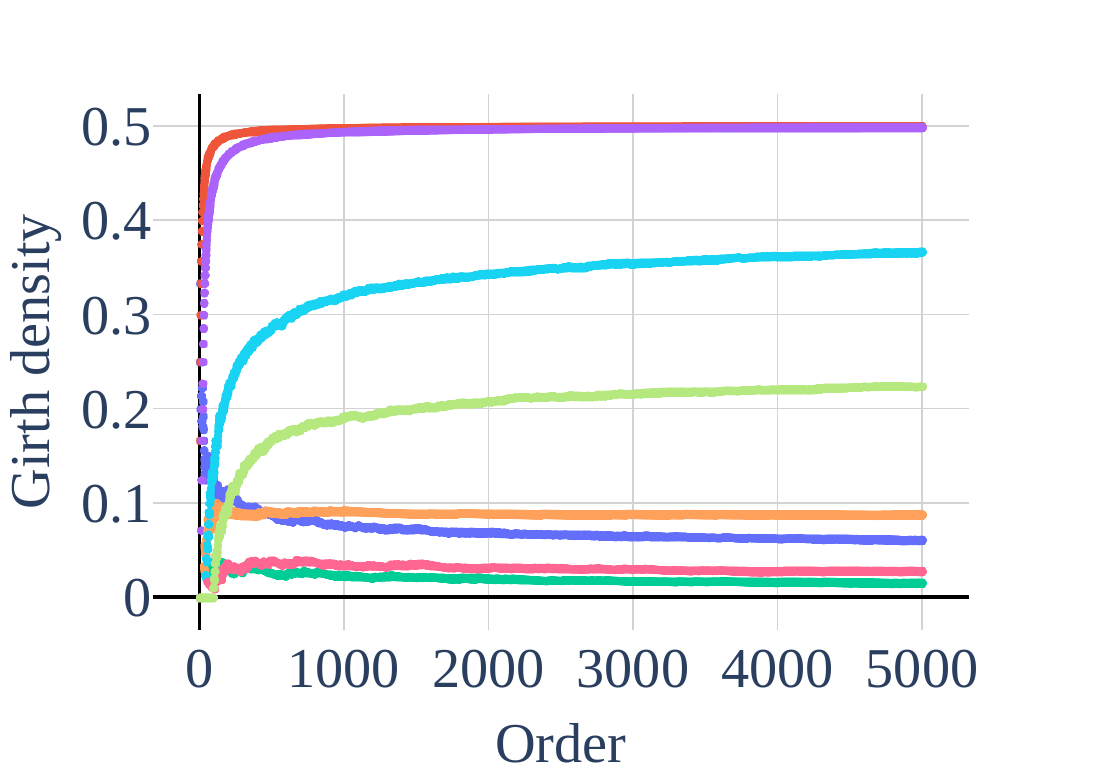} 
    \end{minipage}\hfill
    \begin{minipage}{0.5\textwidth}
        \centering
         \includegraphics[width=\textwidth]{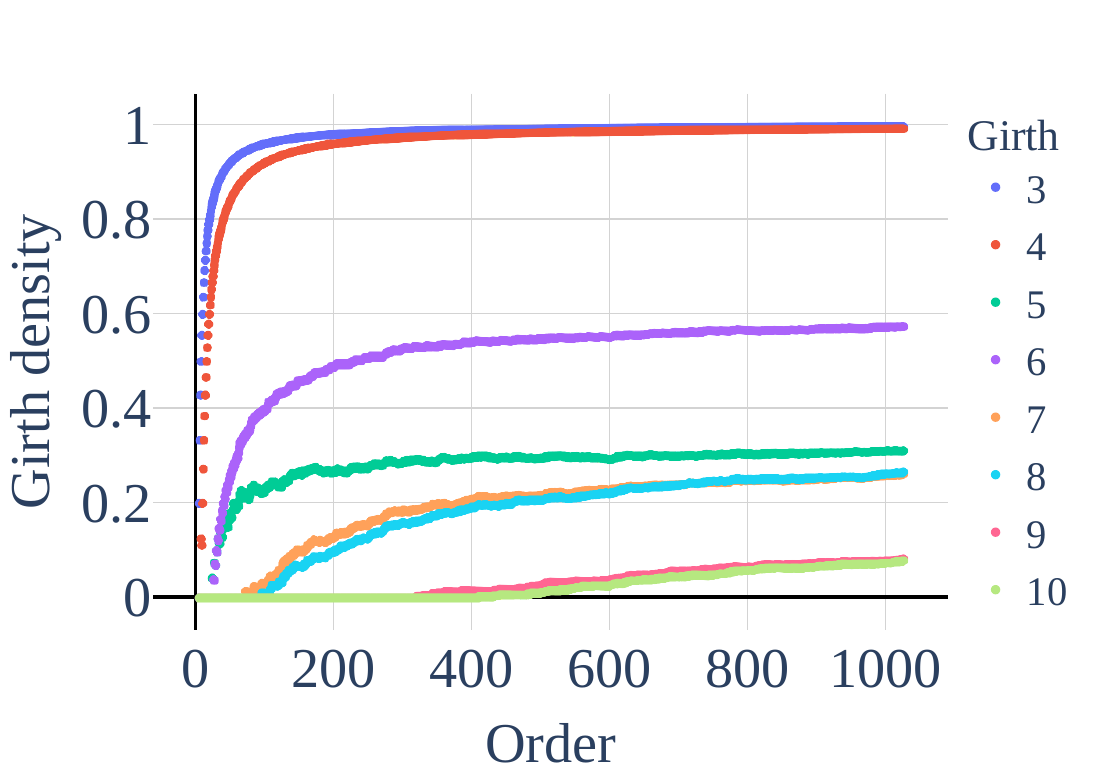} 
    \end{minipage}
    \caption{Girth density of cubic (L) and quartic (R) Cayley graphs.}
    \label{fig:girth_density}
\end{figure}

For small values of $g$, $D(3,g)$ has been determined. Specifically, $D(3,4)=D(3,6)=1/2,D(3,3)=D(3,5)=0$ and $D(3,7)=1/12$. Motivated by the patterns found in Figure \ref{fig:girth_density}, we pose the following question.

\begin{problem}
    For each even positive integer $g$, is $D(3,g)=1/2$?
\end{problem}

\subsection{Bipartiteness}\label{ss:bipartite_data}

A bipartite vertex-transitive graph necessarily has an even number of vertices. Taking this into account, Figure \ref{fig:bip_c} shows the proportion of bipartite cubic and quartic Cayley graphs with $2m\leq n$ vertices. Note that this proportion seems to tend to 1 as $n\to\infty$. Using the collection of small vertex-transitive graphs \cite{HR_2020,HRT_2022}, we also observe a trend of increasing proportions of bipartite graphs. We are prompted to ask the following. 
\begin{problem}
    Let $k\geq 3$ be an integer. Does the proportion of $k$-valent Cayley graphs (or more generally, vertex-transitive graphs) with at most $n$ vertices that are bipartite tend to 1 as $n\to \infty$?
\end{problem}

\begin{figure}[h]
    \centering
       \begin{minipage}{0.5\textwidth}
        \centering
        \includegraphics[width=\textwidth]{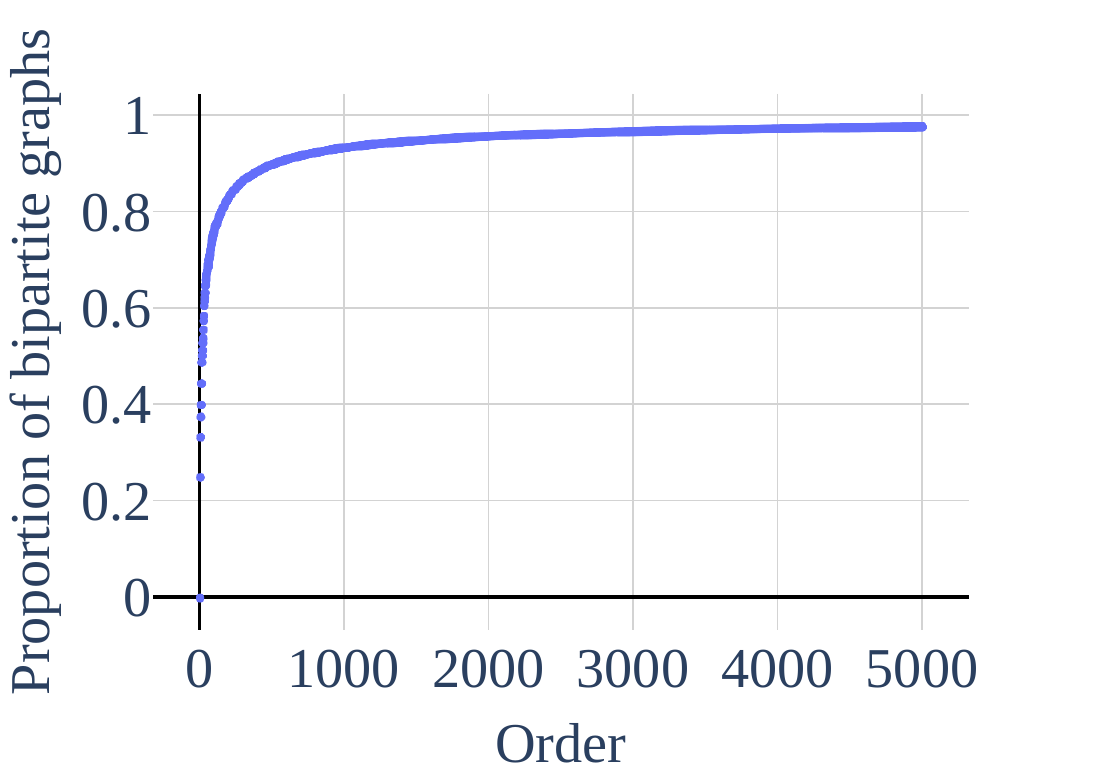} 
    \end{minipage}\hfill
    \begin{minipage}{0.5\textwidth}
        \centering
          \includegraphics[width=\textwidth]{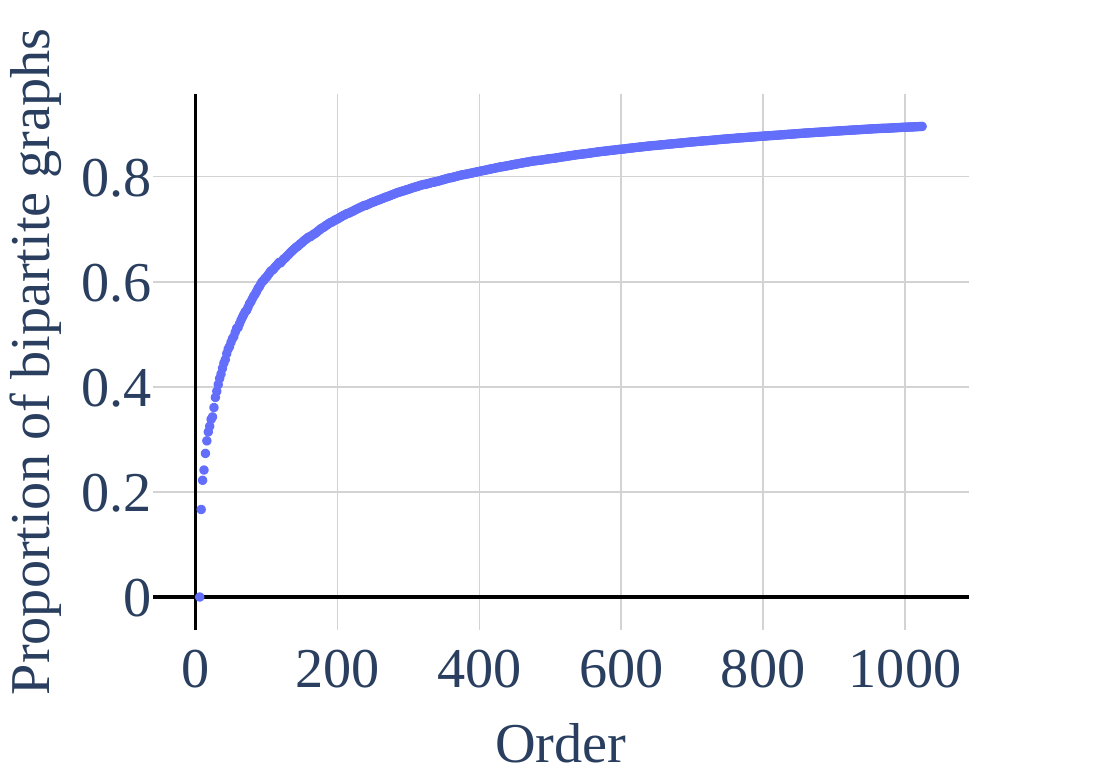} 
    \end{minipage}
     \caption{Cumulative proportion of (even order) cubic (L) and quartic (R) Cayley graphs that are bipartite.}
    \label{fig:bip_c}
\end{figure}

\section*{Acknowledgments}
Both authors acknowledge support of the Slovenian Research and Innovation Agency (ARIS), project numbers P1-0294 and J1-4351.

%